\documentclass[11pt,a4paper,reqno]{amsart}
\usepackage[english]{babel}
\usepackage[applemac]{inputenc}
\usepackage[T1]{fontenc}
\usepackage{palatino}
\usepackage{amsmath}
\usepackage{amssymb}
\usepackage{amsthm}
\usepackage{amsfonts}
\usepackage{mathrsfs}
\usepackage{graphicx}
\usepackage{color}
\usepackage{mathtools}

\usepackage{vmargin}

\setmarginsrb{23mm}{20mm}{23mm}{20mm}{10mm}{10mm}{10mm}{10mm}

\usepackage[colorlinks = true, citecolor = black]{hyperref}
\newtheorem{theorem}{Theorem}[section]

\newtheorem{cor}[theorem]{Corollary}
\newtheorem{prop}[theorem]{Proposition}

\newtheorem{lem}[theorem]{Lemma}

\theoremstyle{definition}
\newtheorem{defn}{Definition}[section]

\newtheorem{ex}[theorem]{Example}

\def\vint{\mathop{\mathchoice%
          {\setbox0\hbox{$\displaystyle\intop$}\kern 0.22\wd0%
           \vcenter{\hrule width 0.6\wd0}\kern -0.82\wd0}%
          {\setbox0\hbox{$\textstyle\intop$}\kern 0.2\wd0%
           \vcenter{\hrule width 0.6\wd0}\kern -0.8\wd0}%
          {\setbox0\hbox{$\scriptstyle\intop$}\kern 0.2\wd0%
           \vcenter{\hrule width 0.6\wd0}\kern -0.8\wd0}%
          {\setbox0\hbox{$\scriptscriptstyle\intop$}\kern 0.2\wd0%
           \vcenter{\hrule width 0.6\wd0}\kern -0.8\wd0}}%
          \mathopen{}\int}

\newcommand{\R}{\mathbb{R}}
\newcommand{\Rn}{{\mathbb R}^n}

\newcommand{\C}{\mathbb{C}}

\newcommand{\Om}{\Omega}
\newcommand{\om}{\omega}
\newcommand{\ud}{\mathrm {d}}

\newcommand{\Hm}{\mathbb{H}^m}
\newcommand{\Hei}{{\mathbb{H}}^{1}}
\newcommand{\He}{\mathbb{H}}

\newcommand{\frk}[1]{\mathfrak{#1}}
\newcommand{\dbd}[2]{\frac{\partial#1}{\partial #2}}

\newcommand{\scr}[1]{\mathscr{#1}}

\definecolor{blau}{rgb}{0.1,0.0,0.9}
\definecolor{violet}{rgb}{0.54, 0.17, 0.89}
\newcommand{\blue}{\color{blau}}

\newcommand{\kom}[1]{}
\renewcommand{\kom}[1]{{\bf \blue /#1/}}

\newcounter{komcounter}
\numberwithin{komcounter}{section}

\def\Xint#1{\mathchoice
	{\XXint\displaystyle\textstyle{#1}}%
	{\XXint\textstyle\scriptstyle{#1}}%
	{\XXint\scriptstyle\scriptscriptstyle{#1}}%
	{\XXint\scriptscriptstyle\scriptscriptstyle{#1}}%
	\!\int}
\def\XXint#1#2#3{{\setbox0=\hbox{$#1{#2#3}{\int}$}
		\vcenter{\hbox{$#2#3$}}\kern-.5\wd0}}

\def\dashint{\Xint-}

\makeatletter
\@namedef{subjclassname@2020}{\textup{2020} Mathematics Subject Classification}
\makeatother

\begin{document}

\title[Harmonic curves from $\Rn$ to $\Hei$]{Harmonic curves from Euclidean domains to Heisenberg group $\Hei$ }
\author[T.\ Adamowicz]{Tomasz Adamowicz}
\address{The Institute of Mathematics, Polish Academy of Sciences \\ ul. \'Sniadeckich 8, 00-656 Warsaw, Poland}
\email{tadamowi@impan.pl}
\author[M.\ Capolli]{Marco Capolli}
\address{Department of Mathematics "Tullio Levi-Civita", Universit\`a degli studi di Padova \\ via Trieste 63, 35121 Padua, Italy}
\email{marco.capolli@unipd.it}
\author[B.\ Warhurst]{Ben Warhurst}
\address{Institute of Mathematics, University of Warsaw \\ ul. Banacha 2, 02-097 Warsaw, Poland}
\email{benwarhurst68@gmail.com}


\begin{abstract}
 We define and study the harmonic curves on domains in $\Rn$ into the first Heisenberg group $\Hei$. These are the $C^2$-regular mappings which are critical points of the second Dirichlet energy and satisfy the weak isotropicity condition. We investigate the geometry of such curves including the comparison and maximum principles,  the Harnack inequalities, the Liouville theorems, the existence results, the Phragm\`en--Lindel\"of theorem, as well as the three spheres theorem.
\end{abstract}

\maketitle

\section{Introduction}
  Harmonic mappings appear naturally in a variety of problems in pure and applied mathematics, for instance, in nonlinear elasticity theory, non-newtonian fluid dynamics, glaciology and cosmology. In pure mathematics harmonic mappings are investigated in differential geometry, in the context of mappings on metric spaces or in relation to differential forms and quasiregular maps.
%
In this work, our focus is directed towards a specific instance of harmonic maps, namely those defined on domains within the Euclidean space with target in the Heisenberg groups $\Hm$ with the emphasize on the first Heisenberg group $\Hei$. There are natural motivations for studying harmonic mappings in the Carnot setting. For example, the fact that Carnot groups, such as Heisenberg groups, have in general curvature, in the sense of Alexandrov, unbounded from above does not allow to directly apply the regularity theory as in~\cite{korschn}. This in turn leads to necessity of developing new techniques and considerations of additional imposed assumptions on mappings, for example the contactivity or isotropic conditions, cf.~\cite{cl}.

Let $\Om\subset \Rn$ be a domain and $f$ be a Sobolev map in $W^{1,2}(\Om, \Hm)$ denoted by $f=(\mathsf{z}, \mathsf{t})$, $\mathsf{z}=\mathsf{x} +i \mathsf{y}$, where $\mathsf{z}: \Om\to \mathbb{R}^{2m}$ and $\mathsf{t}:\Om \to \R$, see Definition~\ref{defn-Sob} and the detailed discussion in Section 2. Then, by Theorems 2.11-2.14 in~\cite{cl} the Korevaar--Schoen energy of $f$ coincides with
$$
E^2(f)= \int_\Omega |\nabla \mathsf{z}|^2 dx,
$$
cf. the discussion at~\eqref{def-E2}. Moreover, the associated Euler--Lagrange system of equations takes the following weak form (see~\eqref{crit-main} and Example 3.1 below):
\begin{equation*}
\int_{\Omega} \sum_{k=1}^{m} \Big (   |\nabla \mathsf{z}^k|^2  \dbd{\xi^\gamma}{x_\gamma} - 2 \sum_{\beta=1}^{n}   \dbd{\mathsf{z}^k}{x_\beta} \overline{ \dbd{\mathsf{z}^k}{x_\gamma}}      \dbd{\xi^\gamma}{x_\beta}  \Big) \, dx=0, \quad \gamma=1,\dots,n,\quad \xi \in C_0^\infty(\Omega,\R^{2n}).
\end{equation*}
However, the geometry of Heisenberg groups strikes back by imposing one more condition on $\mathsf{z}$ namely, the weak isotropicity, see~\eqref{isotrop}. In particular, Theorem 2.17 in~\cite{cl} says that~\eqref{isotrop} allows us to find the non-horizontal component $\mathsf{t}$ of a map $f$ via the contact equations~\eqref{nablat}, see Section 2.
For general $m>1$, the weak isotropicity condition involves minors of the horizontal Jacobi matrix and is, therefore, computationally challenging and difficult in use. It is only the case $m=1$, when condition~\eqref{isotrop} becomes ~\eqref{isotropm1} and allows us to observe that the image of a nontrivial $f$ is a horizontal curve. In a consequence, the component functions $\mathsf{x}, \mathsf{y}$ and $\mathsf{t}$ of $f$ turn out to satisfy elliptic PDEs of the Laplace type, respectively, \eqref{criticaly}, \eqref{eq-f1} and~\eqref{harmonic-t}, see the discussion in Section 3. This in turn enables us to investigate the geometry of such harmonic curves. Our results include:
\begin{itemize}
\item[(1)] the Caccioppoli estimate (Section 4.1),
\item[(2)] the Liouville theorems (Section 4.2),
\item[(3)] the superharmonicity result (Section 4.3),
\item[(4)] the comparison and maximum principles, the Harnack inequality (Section 4.4),
\item[(5)] the existence and uniqueness results (Section 4.5),
\item[(6)] the Phragm\`en--Lindel\"of theorem (Section 4.6),
\item[(7)] the three spheres theorem (Section 4.7).
\end{itemize}


\section{Preliminaries} Following \cite{KORANYI19951}, the Heisenberg group $\He^m$ is given by $\R^{2m+1}$, where we denote coordinates by $({\mathsf{x}},{\mathsf{y}},{\mathsf{t}})$, with ${\mathsf{x}},{\mathsf{y}} \in \R^m$ and ${\mathsf{t}} \in \R$, and the group product is given by the formula
$$({\mathsf{x}},{\mathsf{y}},{\mathsf{t}})({\mathsf{x}}',{\mathsf{y}}',{\mathsf{t}}')=({\mathsf{x}}+{\mathsf{x}}', {\mathsf{y}}+{\mathsf{y}}', {\mathsf{t}}+{\mathsf{t}}'-2 {\mathsf{x}} \cdot {\mathsf{y}}'+2 {\mathsf{y}} \cdot {\mathsf{x}}').$$
If $\{e_i:i=1,\dots, 2m+1\}$ is the standard basis for $\R^{2m+1}$ and $\frk{h}^m$ denotes the Lie algebra where the nontrivial Lie brackets are $[e_i,e_{m+i}]=-4 e_{2m+1}$, $i=1,\dots,m$, then the product above is the Baker--Campbell--Hausdorff model of $\He^m$ built on $\frk{h}^m$.

The left invariant vector fields are framed by
\begin{align*} 
\tilde X_i &= \dbd{}{{\mathsf{x}}_i}+2{\mathsf{y}}_i \dbd{}{{\mathsf{t}}} \quad i=1,\dots , m\\
\tilde Y_i &= \dbd{}{{\mathsf{y}}_i}-2{\mathsf{x}}_i \dbd{}{{\mathsf{t}}}  \quad i=1,\dots , m\\
\tilde T &=\dbd{}{{\mathsf{t}}}
\end{align*}   
for which the only nontrivial brackets are  $$[\tilde X_i,\tilde Y_i] =-4 \tilde T \quad i=1,\dots,m.$$ 

The frame $\{d{\mathsf{x}}_i, d{\mathsf{y}}_i, \theta \}$, where $\theta =d{\mathsf{t}} +2 \sum_i ({\mathsf{x}}_i d{\mathsf{y}}_i -{\mathsf{y}}_i d{\mathsf{x}}_i)$, is the dual of the frame $\{\tilde X_i, \tilde Y_i, \tilde T \}$ and ${\rm ker \, } \theta=\scr{H}$, the horizontal subbundle $\scr{H}={\rm span}\{\tilde X_1,\ldots, \tilde X_m, \tilde Y_1,\ldots, \tilde Y_m\}$. In particular, $\theta$ is a contact form since
$$
(d \theta)^m \wedge \theta =2^{2m} d{\mathsf{x}}_1 \wedge \dots \wedge d{\mathsf{x}}_m \wedge d{\mathsf{y}}_1 \wedge \dots \wedge d {\mathsf{y}}_m \wedge d{\mathsf{t}}. 
$$

The Kor\'anyi metric on $\Hm$ is defined as $d(p, q)=\|p^{-1}q\|$ where
 $$
 \| ({\mathsf{x}},{\mathsf{y}},{\mathsf{t}}) \|= ( (|{\mathsf{x}}|^2+|{\mathsf{y}}|^2)^2 + {\mathsf{t}}^2)^{1/4}.
 $$ 
 Moreover the Haar measure agrees with Lebesgue measure up to a scalar factor.  

Mappings from domains in $\R^n$ to $\He^m$ are natural test cases for the Korevaar--Schoen theory of energy, see \cite{korschn}, with nonriemannian target. When considering variational problems for maps between Riemannian manifolds, and more generally maps from Riemannian manifolds to suitable metric measure spaces $(\mathbb{X}, d, \mu)$, the vital ingredient is a notion of Sobolev space $W^{1,\alpha}(\Omega, \mathbb{X})$, where $\Omega$ is a domain in a Riemannian manifold and $1,\alpha$ indicate that the first order derivatives of the map are $\alpha$-integrable. 

More precisely, the space $L^\alpha (\Omega, \mathbb{X} )$ is defined as the set of Borel-measurable mappings $f : \Omega \to \mathbb{X}$ such that
$$\int_{\Omega} d(f(p),P_0)^\alpha dp < \infty $$
for some point $ P_0 \in \mathbb{X}$. It follows that the set $L^\alpha (\Omega, \mathbb{X} )$ becomes a complete metric space when the distance between $f$ and $g$ is given by
$$
\rho(f,g)=\int_{\Omega} d(f(p),g(p))^\alpha dp, 
$$ 
see 2.4.12 in~\cite{fedr}. 
 The Korevaar--Schoen energy of a map $f \in L^\alpha (\Omega, \mathbb{X} )$ has a variety of specific definitions (not all equivalent), given by the choice of a density function $e_{f,\epsilon}$  and defined as 
\begin{align*}
E^\alpha(f,\Omega)=\sup \left \{ \limsup_{\epsilon \to 0} \int_{\Omega} \phi(p) e_{f,\epsilon}(p) dp : \phi \in C_0(\Omega, [0,1])\right \}. 
\end{align*}
There are two fundamental choices for $e_{f,\epsilon}(p)$, namely the boundary averaging density given by
\begin{align*}
 e_{f,\epsilon}(p)&=\int_{ \partial B_{\epsilon}(p) }  \left (\frac{d(f(p),f(q))}{\epsilon}  \right )^\alpha \frac{d \sigma_\epsilon(q)}{\epsilon^{n-1}}, 
\end{align*}
where $d \sigma_\epsilon(q) $ is the Riemannian surface measure on $\partial B_{\epsilon}(p)$
and the volume averaging density given by
\begin{align*}
e_{f,\epsilon}(p)&=\dashint_{ B_{\epsilon}(p) }  \left (\frac{d(f(p),f(q))}{\epsilon}  \right )^\alpha dq. 
\end{align*}
The volume averaging density is appealing from the metric measure space perspective, however the space of finite energy maps does not resemble a classical Sobolev space, see \cite[Theorem 1.1]{tkh}. 

In the case $\alpha=2$, finite energy maps relative to the boundary averaging density, do form a space resembling a classical Sobolev space. Theorems 2.11, 2.13 and 2.14 in~\cite{cl} prove that if $\Om\subset \Rn$ is a bounded smooth domain and $f=(\mathsf{z}, \mathsf{t})$, $\mathsf{z}=\mathsf{x} +i \mathsf{y}$, then
\begin{itemize}
\item[(i)] If $E^2(f,\Omega)<\infty$ then the function $ \mathsf{z} \in L^2 (\Omega, \R^{2m} )$ is weakly differentiable and $ \mathsf{z} \in W^{1,2}(\Omega, \R^{2m} )$.
\item[(ii)]  If $E^2(f,\Omega)<\infty$ then the function $ \mathsf{t} \in L^{1} (\Omega, \R)$ is weakly differentiable and $(f^*\theta)_x=0$ for a.e. $x \in \Omega$. Moreover
\begin{align}
 \nabla \mathsf{t} = 2( \mathsf{y} \cdot \nabla \mathsf{x} - \mathsf{x} \cdot \nabla \mathsf{y} ) \in L^\beta(\Omega, \R), \quad \beta=\frac{n}{n-1}. \label{nablat} 
\end{align}
\item[(iii)] The energy takes the form
\begin{equation}
 \displaystyle E^2(f,\Omega)= \int_\Omega |\nabla \mathsf{z}|^2 dx. \label{def-E2}
\end{equation} 
\end{itemize}
We stress the fact that we only consider the case $\alpha=2$, instead of the more generic case presented in~\cite{cl}. This is due to the fact that the discussion in~\cite{cl} does not support the cases $\alpha \ne 2$ due to Lemma 2.5 in~\cite{cl} which underpins most of the results in the paper and is only true when $\alpha = 2$, see page 574 in~\cite{korschn} and the discussion following Theorem 1.1 in~\cite{tkh}.

%
%

%

\begin{defn}\label{defn-Sob} The Sobolev space  $W^{1,2}(\Om, \Hm)$ is defined as the set of functions $f=(\mathsf{z}, \mathsf{t})$, $\mathsf{z}=\mathsf{x} +i \mathsf{y}$, such that $E^2(f,\Omega)<\infty$.
\end{defn}

When $f$ is smooth, the condition $f^*\theta=0$ implies that $d(f^*\theta)=f^*d \theta=0$ which means that $f$ is isotropic. Indeed, the symplectic form defined on the horizontal bundle of $\mathbb{H}^m$ is given by 
$$ d \theta=  4 \sum_{k=1}^m d\mathsf{x}^k  \wedge d \mathsf{y}^k$$
and
$$ f^* d \theta=\mathsf{z}^*d \theta = 4  \sum_{\alpha <\beta} \sum_{k=1}^m \left ( \dbd{\mathsf{x}^k}{x_\alpha} \dbd{\mathsf{y}^k}{x_\beta} - \dbd{\mathsf{x}^k}{x_\beta} \dbd{\mathsf{y}^k}{x_\alpha} \right )  dx_\alpha \wedge dx_\beta.$$

The condition that $df_p(T_p\R^n)$ is an isotropic subspace of $T_{f(p)}\R^{2m}$ is exactly $ f^*d \theta=0$ which in our coordinates gives the following conditions:
\begin{align}
\sum_{k=1}^m \left ( \dbd{\mathsf{x}^k}{x_\alpha} \dbd{\mathsf{y}^k}{x_\beta} - \dbd{\mathsf{x}^k}{x_\beta} \dbd{\mathsf{y}^k}{x_\alpha} \right )=0, \quad \alpha<\beta,\, \alpha, \beta=1,\dots, n. \label{isotrop}
\end{align}

If $\mathsf{z}:\Omega \to \C^m$ is smooth, then any smooth function $t:\Omega \to \R$ gives an extension $g=(\mathsf{z}, t):\Omega \to \He^m$ such that $\mathsf{z}^*d \theta=d(g^*\theta)$.  If \eqref{isotrop} holds then $g^*\theta$ is closed and the Poincar\'e lemma implies that locally (on simply-connected sets) there exists a function $h$ such that $d h = g^*\theta$. It follows that $\mathsf{t} = t-h$ gives an extension $f=(\mathsf{z} , \mathsf{t})$ such that $f^*\theta=0$. Theorem 2.16 in~\cite{cl} shows that if $f \in W^{1,2}(\Om, \Hm)$ then $\mathsf{z} : \Omega \to \R^{2m}$ is weakly isotropic
in the sense that \eqref{isotrop} holds a.e. on $\Omega$, and Theorem 2.17 in~\cite{cl} shows that if $\mathsf{z} : \Omega \to \R^{2m}$ is weakly isotropic then there exists a function $\mathsf{t}$ such that $f = (\mathsf{z}, \mathsf{t})  \in W^{1,2}(\Om, \Hm)$. These results are obtained by an indirect smoothing argument similar to \cite{dairb}.

%


\section{Euler--Lagrange system of equations}	

 Theorem~4.1 of \cite{cl} shows that if $\Om\subset \Rn$ is s bounded domain, then the Euler--Lagrange system for the variational equation $\frac{d}{ds} E^2(\mathsf{z}_s)|_{s=0}=0$, where $\mathsf{z}_s=\mathsf{z}( x + s \xi(x))$ and  $\xi \in C_0^\infty(\Omega,\R^{n})$, takes the following weak form:
\begin{equation}\label{crit-main}
\int_{\Omega} \sum_{k=1}^{m} \Big (   |\nabla \mathsf{z}^k|^2  \dbd{\xi^\gamma}{x_\gamma} - 2 \sum_{\beta=1}^{n}   \dbd{\mathsf{z}^k}{x_\beta} \overline{ \dbd{\mathsf{z}^k}{x_\gamma}}      \dbd{\xi^\gamma}{x_\beta}  \Big) \, dx=0, \quad \gamma=1,\dots,n.
\end{equation}
In particular, the equations above are for variations of the form $\xi=(0, \dots,\xi^\gamma, \dots,0)$, and so setting $\xi^\gamma=\phi$ for any test function  $\phi \in C_0^\infty(\Omega,\R)$, we have the following equations
\begin{align*}
\int_{\Omega} \sum_{k=1}^{m} \Big (|\nabla \mathsf{z}^k|^2  \dbd{\phi}{x_\gamma} - 2 \sum_{\beta=1}^{n}   \dbd{\mathsf{z}^k}{x_\beta} \overline{ \dbd{\mathsf{z}^k}{x_\gamma}}      \dbd{\phi}{x_\beta}  \Big) \, dx=0, \quad \gamma=1,\dots,n.
\end{align*}
\begin{ex}

Let us analyze system~\eqref{crit-main} in the case of $f:\Omega\to\Hei$ a Sobolev map in $W^{1,2}(\Om, \Hei)$, as in Definition~\ref{defn-Sob}. Then $f(x_1,\dots,x_{n})=(\mathsf{x},\mathsf{y},\mathsf{t})$ with $\mathsf{x},\mathsf{y},\mathsf{t}:\Omega\to\R$. Moreover, with the slight abuse of notation, let us set $\mathsf{z}=(\mathsf{x}(x_1,\ldots,x_n),\mathsf{y}(x_1,\ldots,x_n))$, the horizontal part of map $f$. 

Then, an equivalent formulation of system~\eqref{crit-main}, perhaps more convenient to study, is as follows.

\begin{equation*}
\begin{cases}
	\displaystyle{\int_{\Omega} \left(|D\mathsf{z}|^2-2 \left|\dbd{\mathsf{z}}{x_1}\right|^2\right) \dbd{\phi^1}{x_1}-2 \left(\dbd{\mathsf{z}}{x_1}\cdot \dbd{\mathsf{z}}{x_2}\right)\dbd{\phi^1}{x_2} - \dots - 2 \left(\dbd{\mathsf{z}}{x_1}\cdot \dbd{\mathsf{z}}{x_n}\right)\dbd{\phi^1}{x_n}\,\ud x=0}\vspace*{.1cm}\\
	\displaystyle{\int_{\Omega} -2 \left(\dbd{\mathsf{z}}{x_2}\cdot \dbd{\mathsf{z}}{x_1}\right)\dbd{\phi^2}{x_1} + \left(|D\mathsf{z}|^2-2 \left|\dbd{\mathsf{z}}{x_2}\right|^2\right)\dbd{\phi^2}{x_2}- \dots - 2 \left(\dbd{\mathsf{z}}{x_2}\cdot \dbd{\mathsf{z}}{x_n}\right)\dbd{\phi^2}{x_n}\,\ud x=0}\\
	\hspace*{5cm} \vdots\\
	\displaystyle{\int_{\Omega}-2 \left(\dbd{\mathsf{z}}{x_n}\cdot \dbd{\mathsf{z}}{x_1}\right)\dbd{\phi^n}{x_1} - \dots -2 \left(\dbd{\mathsf{z}}{x_n}\cdot \dbd{\mathsf{z}}{x_{n-1}}\right)\dbd{\phi^n}{x_{n-1}}+\left(|D\mathsf{z}|^2-2 \left|\dbd{\mathsf{z}}{x_n}\right|^2\right)\dbd{\phi^n}{x_n}\,\ud x=0},
\end{cases}
\end{equation*}
where $\phi=(\phi^1,\dots,\phi^{n})\in C_0^{\infty}(\Omega,\Rn)$ is a test mapping.
\end{ex}


Setting $A=\sum_{k=1}^{m}   |\nabla \mathsf{z}^k|^2 $ and $B_{\gamma,\beta}= - 2   \sum_{k=1}^{m} \dbd{\mathsf{z}^k}{x_\beta} \overline{ \dbd{\mathsf{z}^k}{x_\gamma}}$ brings equations~\eqref{crit-main} into the form
\begin{align*} 
 \int_{\Omega}  \sum_{\beta=1}^{n} (\delta_{\gamma,\beta} A + B_{\gamma,\beta}   )\dbd{\phi}{x_\beta} dx=0, \quad \gamma=1,\dots,n. 
\end{align*}
 Note that the isotropy conditions \eqref{isotrop} imply that  $$B_{\gamma,\beta}=- 2   \sum_{k=1}^{m} \dbd{\mathsf{x}^k}{x_\beta}  \dbd{\mathsf{x}^k}{x_\gamma} + \dbd{\mathsf{y}^k}{x_\beta}  \dbd{\mathsf{y}^k}{x_\gamma}.$$

Next we set $F_\gamma^\beta= \delta_{\gamma,\beta} A+B_{\gamma, \beta}$ and $F_\gamma=(F_\gamma^1, \dots, F_\gamma^{n})$ so that the previous equation becomes
$$
  \int_{\Omega}  \langle  F_\gamma , \nabla \phi \rangle \, dx=0, \quad \gamma=1,\dots,n,
 $$
or
$$
  \int_{\Omega}  {\rm div} F_\gamma \, \phi  \, dx=0, \quad \gamma=1,\dots,n,
$$
at least in the distributional sense. If $f=(\mathsf{z}, t)$ is $C^2$ and a critical point of the energy $E^2$ on $\Omega$, then ${\rm div} F_\gamma=0$ on $\Omega$. Moreover, direct calculation shows that
\begin{align}
{\rm div} F_\gamma &= -2 \sum_k \Big(  \Delta \mathsf{x}^k  \dbd{\mathsf{x}^k}{x_\gamma} + \Delta \mathsf{y}^k \dbd{\mathsf{y}^k}{x_\gamma}    \Big ), \quad \gamma=1,\dots,n. \label{divfm}
\end{align}

%
%

In the specific case $m=1$, the condition that $df_p(T_p\R^n)$ is an isotropic subspace of $T_{f(p)}\R^{2}$ implies ${\rm dim}\, df_p(T_p\R^n) \leq 1$, hence the image $f(\Omega)$ is a point or a horizontal curve. The isotropy conditions \eqref{isotrop} become
\begin{align}
\dbd{\mathsf{x}}{x_\alpha} \dbd{\mathsf{y}}{x_\beta} - \dbd{\mathsf{x}}{x_\beta} \dbd{\mathsf{y}}{x_\alpha} =0,\quad  \alpha<\beta,\, \alpha, \beta=1,\dots, n. \label{isotropm1}
\end{align}

Let $\mathsf{y}:\Omega \to \R$ be any $C^1$ function. If $I$ is an interval containing the image of $\mathsf{y}$ and $H : I \to \R$ is any $C^2$ function,  then $\mathsf{x} + i \mathsf{y}$, where $\mathsf{x}=H\circ \mathsf{y}$, defines a $C^1$ solution of \eqref{isotropm1} on $\Omega$ such that $\nabla \mathsf{x} =(H' \circ \mathsf{y})  \nabla \mathsf{y}$. Furthermore, the discussion at the end of the previous section shows that $ \mathsf{z}=H\circ \mathsf{y}+ i \mathsf{y}$ extends locally to a map $f=(\mathsf{z}, \mathsf{t})$ into $\He^1$ with the property $f^*\theta=0$.  We note that \eqref{nablat} implies that ${\rm span}\,\{\nabla \mathsf{x}, \nabla \mathsf{y}, \nabla \mathsf{t}\}= {\rm span}\,\{\nabla \mathsf{y}\}$ pointwise, hence confirming that $f=(\mathsf{z},  \mathsf{t})$ parametrizes a point or a horizontal curve.

If the $f$ is not constant, we can choose $\mathsf{y}$ so that $\dbd{\mathsf{y}}{x_\gamma}$ is not identically zero for some particular index $\gamma$, and set $\mathsf{x}=H\circ \mathsf{y}$ as above. Then
\begin{align*}
{\rm div} F_\gamma &= -2  \Big(  H''(\mathsf{y}(x)) H'(\mathsf{y}(x)) |\nabla \mathsf{y}|^2  + \Big (1 + H'(\mathsf{y}(x))^2     \Big ) \Delta \mathsf{y}  \Big )\dbd{\mathsf{y}}{x_\gamma}.
\end{align*}
Hence, if $\mathsf{z}=H\circ \mathsf{y} + i \mathsf{y}$, then $f=(\mathsf{z}, \mathsf{t})$ is a critical point of $E^2$ if and only if
\begin{align}
\Delta \mathsf{y}  &= - \frac{(H''\circ \mathsf{y}) (H'\circ \mathsf{y})}{1 + (H' \circ \mathsf{y})^2 } |\nabla \mathsf{y}|^2 = - r(\mathsf{y}, H', H'') |\nabla \mathsf{y}|^2. \label{criticaly}
\end{align} 

Let $m>1$ and for each $k=1,\dots, m$, let $\mathsf{z}^k =\mathsf{x}^k + i \mathsf{y}^k$ be a solution to the isotropy equations \eqref{isotropm1}, where $\mathsf{y}^k:\Omega \to \R$ is any $C^1$ function and  $\mathsf{x}^k=H_k\circ \mathsf{y}^k$  where $H_k$ is any $C^2$ function defined on the image of $\mathsf{y}^k$.  If $ \mathsf{z}=(\mathsf{z}^1, \dots, \mathsf{z}^m)$, then as discussed in the paragraph following equation $\eqref{isotrop}$, there exists locally a function $\mathsf{t}$ such that $f=(\mathsf{z} , \mathsf{t})$ satisfies $f^*\theta=0$. Furthermore, \eqref{divfm} shows that $f$ is critical if each $\mathsf{z}^k$ is the complex part of a critical map from $\Omega$ to $\Hei$. In particular, we stress the fact that a function $\mathsf{y}^k:\Omega\to \R$ needs not to depend on all the variables.

Given that $\mathsf{x}=H\circ\mathsf{y}$, we have by direct calculation the following two identities 
\[
\nabla \mathsf{x}= \nabla (H\circ \mathsf{y})=(H'\circ \mathsf{y})\nabla \mathsf{y},\quad\hbox{ and }\quad \Delta \mathsf{x}=(H''\circ \mathsf{y}) |\nabla \mathsf{y}|^2+(H'\circ \mathsf{y}) \Delta \mathsf{y},
\]
which together with~\eqref{criticaly} imply that $\mathsf{x}$ satisfies the following equation:
\begin{equation}\label{eq-f1}
\Delta \mathsf{x}=\frac{H''\circ \mathsf{y}}{1+(H'\circ \mathsf{y})^2}|\nabla \mathsf{y}|^2.
\end{equation}

As for the third component function $\mathsf{t}$, the contact condition implies that the following equation is satisfied by $\mathsf{t}$:
\begin{equation*}
\nabla \mathsf{t}= 2\mathsf{y}\,\nabla \mathsf{x}-2\mathsf{x}\,\nabla \mathsf{y}=2\big(\mathsf{y} (H'\circ \mathsf{y})-H\circ \mathsf{y}\big)\nabla \mathsf{y}. 
\end{equation*}
Hence
\begin{align*} 
\Delta \mathsf{t} &=2 \langle \nabla \big(\mathsf{y} (H'\circ \mathsf{y})-H\circ \mathsf{y}\big) , \nabla \mathsf{y} \rangle + 2\big(\mathsf{y} (H'\circ \mathsf{y})-H\circ \mathsf{y}\big)  \Delta \mathsf{y}\\
&=2 \mathsf{y} (H''\circ \mathsf{y}) |\nabla \mathsf{y}|^2 + 2\big(\mathsf{y} (H'\circ \mathsf{y})-H\circ \mathsf{y}\big)  \Delta \mathsf{y}.
\end{align*}
Moreover, if $f=(\mathsf{z}, \mathsf{t})$ is critical then by~\eqref{criticaly} we have 
\begin{align*}
 \Delta \mathsf{t} &= 2 \mathsf{y} (H''\circ \mathsf{y}) |\nabla \mathsf{y}|^2 - 2 \big (\mathsf{y}(H' \circ \mathsf{y})- \, H \circ \mathsf{y}  \big )  \frac{(H' \circ \mathsf{y})(H''\circ \mathsf{y})} {1 + (H' \circ \mathsf{y})^2} |\nabla \mathsf{y}|^2 
\end{align*}
which in the case $H$ is not constant shows that
\begin{align}
\Delta \mathsf{t} &= 2\frac{\left(\mathsf{y}+(H \circ \mathsf{y})(H' \circ \mathsf{y})\right)(H''\circ \mathsf{y})} {1 + (H' \circ \mathsf{y})^2} |\nabla \mathsf{y}|^2. \label{harmonic-t}
\end{align}

\section{Geometry of the harmonic curves from $\Rn$ to $\Hei$}

From now on, unless specified differently, we will investigate the \emph{harmonic curves}, i.e. $C^2$ critical points of the energy $E^2$, cf.~\eqref{def-E2}, 
$$
f=(\mathsf{x},\mathsf{y},\mathsf{t}): \Om \to \Hei,
$$
 defined on a domain $\Om\subset \Rn$ and satisfying the weak isotropicity condition~\eqref{isotropm1} for $m=1$. We will also use the term \emph{$C^2$-harmonic mapping} in order to emphasize relations of the so-defined harmonic curves to the critical points of the energy $E^2$.

\subsection{The Caccioppoli estimate}
We now derive a classical important energy estimate for solutions of~\eqref{criticaly}.  The weak formulation of \eqref{criticaly} reads as:
\begin{equation}
\int_\Omega \langle \nabla \mathsf{y}, \nabla \phi \rangle=  \int_\Omega \frac{(H''\circ \mathsf{y}) (H'\circ \mathsf{y})}{1 + (H' \circ \mathsf{y})^2 } |\nabla \mathsf{y}|^2 \phi \label{eq-k-missing}
\end{equation}
for any test function $\phi \in C_0^\infty(\Omega, \R)$.

\begin{lem}
 Let $\Om\subset \Rn$ be a bounded domain and $f=(H\circ \mathsf{y} + i\mathsf{y} ,\mathsf{t})$ be a harmonic curve. If 
 \begin{equation}\label{est-Cac-cond}
 \sup_\Omega \Big |\frac{(H''\circ \mathsf{y}) (H'\circ \mathsf{y})}{1 + (H' \circ \mathsf{y})^2 } \, \mathsf{y}  \Big |:=\mu < 1,
 \end{equation}
 then there exists a positive constant $C(\mu)$ such that the following Caccioppoli-type estimate holds for any test function $\eta\in C_0^\infty(\Om)$
\begin{equation}\label{est-Cac}
 \int_\Omega |\nabla \mathsf{y} |^2 \eta^2 \leq C(\mu) \int_\Omega |\nabla \eta|^2 \mathsf{y}^2.
 \end{equation} 
\end{lem}

Before presenting the proof of the estimate, let us analyze the condition~\eqref{est-Cac-cond}. It is easy to see that, by direct differentiation, one gets the following
\[
\frac{(H''\circ \mathsf{y}) (H'\circ \mathsf{y})}{1 + (H' \circ \mathsf{y})^2 } \, \mathsf{y}=  \frac{s}{2} \frac{\ud}{\ud s} \left(\ln(1+(H'(s))^2\right)\Big|_{s=\mathsf{y}},
\] 
and, moreover, condition~\eqref{est-Cac-cond} holds, for instance if $|(H''\circ \mathsf{y})\,(H' \circ \mathsf{y})|\leq C<1$ in $\Om$. 
\begin{proof}
In the proof we follow the standard approach in obtaining the energy estimate for elliptic PDEs. Namely, let $\eta \in C_0^\infty (\Omega ,\R)$ and define the following test function $\phi:=\eta^2 \mathsf{y}$. Then $\nabla \phi = 2\eta \,  \mathsf{y} \nabla \eta  + \eta^2\nabla \mathsf{y}$ which upon testing \eqref{eq-k-missing} with $\phi$ implies that  
 \begin{align*}
  \int_\Omega \langle \nabla \mathsf{y},  2 \eta \, \mathsf{y}  \, \nabla \eta+ \eta^2 \nabla \mathsf{y}  \rangle &= \int_\Omega  \frac{(H''\circ \mathsf{y}) (H'\circ \mathsf{y})}{1 + (H' \circ \mathsf{y})^2 } |\nabla \mathsf{y}|^2 \,\eta^2 \mathsf{y}. 
 \end{align*}
  Therefore,
 \begin{align} 
 \int_\Omega |\nabla \mathsf{y} |^2 \eta^2 &= \int_\Omega   r(\mathsf{y}, H', H'') |\nabla \mathsf{y}|^2 \,\eta^2 \mathsf{y}  -  2 \int_\Omega \langle \nabla \mathsf{y}, \nabla \eta  \rangle\eta \mathsf{y}  \nonumber \\
 &\leq  \int_\Omega   \Big | r(\mathsf{y}, H', H'') \, \mathsf{y} \Big |  |\nabla \mathsf{y}|^2 \,\eta^2   +  2 \int_\Omega |  \eta \nabla \mathsf{y}| \,  |\mathsf{y} \nabla \eta   | \nonumber  \\
 &\leq  \int_\Omega   \Big | r(\mathsf{y}, H', H'') \, \mathsf{y} \Big |  |\nabla \mathsf{y}|^2 \,\eta^2   +   \int_\Omega \lambda^2  | \nabla \mathsf{y}|^2 |  \eta|^2 +  \int_\Omega \frac{1}{\lambda^2} |\mathsf{y} |^2 | \nabla \eta   |^2, \label{caccpolineq0}
 \end{align} 
 where in the last step we used the classical inequality $2xy \leq \lambda^2x^2+\frac{1}{\lambda^2}y^2$. It now follows that 
 \begin{align*} 
 \lambda^2 \int_\Omega \Big ( 1 -    \lambda^2  -     | r(\mathsf{y}, H', H'') \, \mathsf{y}  |   \Big ) |\nabla \mathsf{y}|^2 \,\eta^2 
 &\leq   \int_\Omega \mathsf{y}^2 | \nabla \eta   |^2
 \end{align*} 
 which shows that when $\mu=\sup_\Omega | r(\mathsf{y}, H', H'') \, \mathsf{y}  |<1$, we have the following Caccioppoli-type estimate
 \begin{align*} 
 \Big ( 1 -    \lambda^2  - \mu^2    \Big ) \lambda^2 \int_\Omega  |\nabla \mathsf{y}|^2 \,\eta^2 
 &\leq   \int_\Omega \mathsf{y}^2 | \nabla \eta   |^2
 \end{align*} for all $\lambda$ satisfying $\lambda^2 < 1-\mu^2$. 
 Alternatively, we can apply the Cauchy-Schwarz/the H\"older inequality at the penultimate line in  \eqref{caccpolineq0} and rearrange to get 
 \begin{align*} 
  \int_\Omega  |\nabla \mathsf{y}|^2 \,\eta^2  &\leq  \frac{4}{(1-\mu)^2}   \int_\Omega | \nabla \eta   |^2 \mathsf{y}^2. 
 \end{align*} 
\end{proof}

\subsection{The Liouville theorems} As the immediate corollary from the Caccioppoli inequality we get the following Liouville type theorem for $C^2$-harmonic curves.

\begin{cor} Let $\Om = \Rn$ and $f=(H\circ \mathsf{y} + i\mathsf{y} ,\mathsf{t})$ be a harmonic curve such that 
$$
\sup_\Omega \left |\frac{(H''\circ \mathsf{y}) (H'\circ \mathsf{y})}{1 + (H' \circ \mathsf{y})^2 } \, \mathsf{y}  \right | < 1
$$ 
and $\mathsf{y}$ is bounded in $\Rn$. Then $f$ is constant.
\end{cor}
\begin{proof}
 Let $\eta \in C_0^{\infty}(B_{2\rho })$ satisfy $\eta\equiv 1$ on the ball $B_\rho$  and $|\nabla \eta|^2\leq \frac{c}{\rho^{n+1}}$ for any positive radius $\rho$. Applying the estimate~\eqref{est-Cac} for the test function $\eta$ we have 
 \begin{equation*}
  \int_{B_\rho} |\nabla \mathsf{y} |^2 \leq \int_{B_{2\rho}} |\nabla \mathsf{y} |^2 \eta^2 \leq   \int_{B_{2\rho}} |\nabla \eta |^2 \mathsf{y}^2 \lesssim \|\mathsf{y} \|^2_\infty \int_{B_{2\rho}}|\nabla \eta|^2 \lesssim_{\om_n} \|\mathsf{y} \|_\infty^2 \frac{c}{\rho},
 \end{equation*} 
 where $\om_n$ stands for the volume of the unit ball in $\Rn$. Letting $\rho\to \infty$ shows that that $|\nabla \mathsf{y}| = 0$ at every point of $\Rn$ which implies that $\mathsf{y}$ must be constant. It follows from \eqref{nablat} that $\nabla \mathsf{t}=0$ at every point of $\Rn$ which implies that $\mathsf{t}$ is constant and consequently $f$ is constant. 
\end{proof}

 Even though the main subject of our interest are $C^2$-harmonic mappings, by using the monotonicity formula from~\cite{cl}, we can still prove the following variant of the Liouville theorem for Sobolev maps with values in $\Hm$.

\begin{prop}
 Let $f\in W_{loc}^{1,2}(\Rn, \Hm)$ such that $f=(\mathsf{z},\mathsf{t})$ be a weak solution to the Euler--Lagrange system of equations~\eqref{crit-main} satisfying 
 \[
 \lim_{r\to \infty} \frac{1}{r^{n-2}}\int_{B(0,r)} |\nabla \mathsf{z}|^2 =0.
 \]
 Then $|\nabla \mathsf{z}|\equiv 0$ in $\Rn$.
 
 Similarly, if $f\in W^{1,2}(\Rn, \Hm)$ such that $f=(\mathsf{z},\mathsf{t})$ is a weak solution to the Euler--Lagrange system of equations~\eqref{crit-main}, then $|\nabla \mathsf{z}|\equiv 0$ in $\Rn$.
\end{prop}

\begin{proof}
  Since $f$ is a local solution to~\eqref{crit-main}, it is in particular a solution on any ball $B(0,r)\subset \Rn$. We can therefore appeal to the following monotonicity formula in Corollary 4.2 in~\cite{cl} holding for all $0<r<R$:
  \[
   \frac{1}{r^{n-2}}\int_{B(0,r)} |\nabla \mathsf{z}|^2 \leq  \frac{1}{R^{n-2}}\int_{B(0,R)} |\nabla \mathsf{z}|^2 \to 0, \hbox{ as } R\to \infty.
  \]
  Hence, for all $r>0$ it holds that $\int_{B(0,r)} |\nabla \mathsf{z}|^2\equiv 0$ and the first assertion follows. Similarly, the same monotonicity formula implies that if we strengthen the integrability assumption and require that $f\in W^{1,2}(\Rn, \Hm)$, then
  \[
  \frac{1}{r^{n-2}}\int_{B(0,r)} |\nabla \mathsf{z}|^2 \leq  \frac{1}{R^{n-2}}\int_{B(0,R)} |\nabla \mathsf{z}|^2\leq \frac{\|f\|_{W^{1,2}(\Rn, \Hm)}}{R^{n-2}}\to 0,\hbox{ as } R\to \infty
   \]
   and the second assertion follows as well.
\end{proof}

\subsection{Superharmonicity result}

It turns out that $|z|$, the modulus of the horizontal part $z$ of a harmonic map $f$, is a supersolution to the second order quasilinear elliptic PDE.

\begin{prop}\label{prop-supersol}
	Let $\Om\subset \Rn$ be an open set and $f=(\mathsf{z},t):\Om\to \Hei$ be a harmonic curve, $\mathsf{z}=H\circ \mathsf{y} + i\mathsf{y}$, then the function $u:=|\mathsf{z} |$ satisfies the following differential inequality:
	\begin{equation}\label{subharm-ineq}
	\Delta u + \frac{|\nabla u|^2}{u}\geq 0,
	\end{equation}
	at points in $\Om$ where $|\mathsf{z}|\not=0$, provided that $H$ satisfies
	\begin{equation}
	  H'' (H - {\rm Id} H') \geq 0, \label{cond-subh-H} 
	\end{equation}
	on the range of $\mathsf{y}$. In particular, it holds if $H$ is nonnegative and $H'\leq 0, H''\geq 0$.
\end{prop}
\begin{proof}
	Since $f$ is a harmonic curve, we have that $\mathsf{z}=H\circ \mathsf{y} + i \mathsf{y}$ for some $C^2$ function $H$ defined on the range of $\mathsf{y}$. By direct calculation
	\begin{align*}
	\nabla|\mathsf{z}| &=\frac{1}{|\mathsf{z}|}\left[(H' \circ \mathsf{y})  (H \circ \mathsf{y}) + \mathsf{y}\right]\nabla \mathsf{y} := \frac{q}{|\mathsf{z}|} \nabla \mathsf{y}
	\end{align*}
	and
	\begin{align*}
	\Delta |\mathsf{z}| &= {\rm div}\left( \frac{q}{|\mathsf{z}|} \nabla \mathsf{y} \right) \\
	&= \big \langle \nabla \frac{q}{|\mathsf{z}|}, \nabla \mathsf{y} \big \rangle +\frac{q}{|\mathsf{z}|} \Delta \mathsf{y} \\
	& = \big \langle \frac{1}{|\mathsf{z}|} \nabla q - \frac{q}{|\mathsf{z}|^2} \nabla |\mathsf{z}|, \nabla \mathsf{y} \big\rangle -\frac{qr}{|\mathsf{z}|} \nabla |\mathsf{y}|^2\\
	&= \big \langle \frac{s}{|\mathsf{z}|}  \nabla \mathsf{y} - \frac{q^2}{|\mathsf{z}|^3} \nabla |\mathsf{y}|, \nabla \mathsf{y} \big \rangle - \frac{qr}{|\mathsf{z}|} \nabla |\mathsf{y}|^2\\
	&= \big ( \frac{s}{|\mathsf{z}|}  - \frac{q^2}{|\mathsf{z}|^3}  -\frac{qr}{|\mathsf{z}|} \big ) | \nabla \mathsf{y}|^2,
	\end{align*}
 where $$s= (H \circ \mathsf{y}) (H'' \circ \mathsf{y}) + (H' \circ \mathsf{y})^2 +1$$ and  $r$ is the function defined at \eqref{criticaly}, i.e. $r=\frac{(H''\circ \mathsf{y}) (H'\circ \mathsf{y})}{1 + (H' \circ \mathsf{y})^2 }$. It follows that
$$
|\mathsf{z}| \Delta |\mathsf{z}|+ |\nabla|\mathsf{z}||^2 = (s-qr) |\nabla \mathsf{y}|^2 \geq 0
$$
provided that $s-qr \geq 0$. The proof is completed upon noticing that this inequality holds by assumption~\eqref{cond-subh-H}, as  
\begin{align*}
s-qr&=1+ (H' \circ \mathsf{y})^2 + (H \circ \mathsf{y}) (H'' \circ \mathsf{y}) - ((H' \circ \mathsf{y})  (H \circ \mathsf{y}) + \mathsf{y})\frac{(H''\circ \mathsf{y}) (H'\circ \mathsf{y})}{1 + (H' \circ \mathsf{y})^2 } \\
&= \frac{1}{1+ (H' \circ \mathsf{y})^2} \bigg \{(1+ (H' \circ \mathsf{y})^2)^2 + (H''\circ \mathsf{y}) \big[H\circ \mathsf{y} - \mathsf{y} (H'\circ \mathsf{y})\big]\bigg\} \geq 0.
\end{align*} 
\end{proof}

It turns out that Proposition~\ref{prop-supersol} allows us to show the comparison principle for the moduli of horizontal parts of harmonic mappings.
\begin{cor} 
 Let $f$ and $g$ be harmonic curves on a domain $\Om\subset \Rn$ such that $f, g\in C^{0}(\overline{\Om})$.  Assume $|\mathsf{z}(f)|$ satisfies~\eqref{subharm-ineq}, while $|\mathsf{z}(g)|$ satisfies the opposite inequality (i.e. $\leq 0$ in~\eqref{subharm-ineq}). 
 
 If $|\mathsf{z}(f)|\leq |\mathsf{z}(g)|$ on $\partial \Om$, then $|\mathsf{z}(f)|\leq |\mathsf{z}(g)|$ in $\Om$ provided that $\mathsf{z}(f)\not=0$ and $\mathsf{z}(g)\not=0$ in $\Om$.
\end{cor}
 \begin{proof}
 The assumptions imply that
 \[
 \Delta |\mathsf{z}(f)| + \frac{|\nabla |\mathsf{z}(f)| |^2}{|\mathsf{z}(f)|}\geq 0 \geq  \Delta |\mathsf{z}(g)| + \frac{|\nabla |\mathsf{z}(g)| |^2}{|\mathsf{z}(g)|}
 \]
 and the proposition follows from Theorem 10.1 in~\cite{gt} for the second order quasilinear operator 
 $$
 Q(u):=\Delta u+ \frac{|\nabla u|^2}{u}.
 $$
 In particular, assumptions (i), (ii) and (iv) of Theorem 10.1 in \cite{gt} trivially hold since in our case the principal part of $Q$ is the Laplacian. For the readers convenience we briefly discuss verification of the assumptions of the theorem:
 \begin{itemize}
 \item[-] $Q$ is uniformly elliptic, giving (i),
 \item[-] the principal part of $Q$ has constant coefficients independent of a solution, hence (ii) holds,
 \item[-] the lower order terms coefficient $b$ of the operator $Q$, defined on $\Om\times\R\times \Rn$ as $b(x,u,v) =|v|^2/u$, is continuously differentiable with respect to the $v$ variable, thus (iv) holds.
 \end{itemize}
   For condition (iii) in assumptions of Theorem 10.1, we note that function $b$ is a non-increasing function of $u \in \R_{+}$ for each fixed $(x,v) \in \Om\times \Rn$, as we assume that $u\geq 0$. 
\end{proof}

\subsection{The comparison and maximum principles, the Harnack inequality}

As in the previous section, Theorem 10.1 in~\cite{gt} can be applied to the operator $Qu=\Delta u+ r(\mathsf{y}, H', H'') |\nabla u|^2$ to give a comparison principle for the component functions of harmonic curves.

\begin{prop} \label{comp-pr}
 Let $f=(\mathsf{x}(f),\mathsf{y}(f),\mathsf{t}(f))$ and $g=(\mathsf{x}(g),\mathsf{y}(g),\mathsf{t}(g))$ be harmonic curves on a domain $\Om \subset \R^n$ corresponding to a given $H$ and such that $f,g \in C^{0}(\overline{\Om})$. If $\mathsf{y}(f)\leq \mathsf{y}(g)$ on $\partial \Om$, then $\mathsf{y}(f)\leq \mathsf{y}(g)$ in $\Om$ provided that $H$ satisfies the following ODI on the range of $\mathsf{y}(f)$:
 \begin{equation}\label{cond-comp2}
  H''' H'(1+(H')^2)+(H'')^2(1-(H')^2)\leq 0.
 \end{equation}
 Moreover, if $H$ is additionally non-increasing then also $\mathsf{x}(f)\leq \mathsf{x}(g)$ on $\partial \Om$ implies that $\mathsf{x}(f)\leq \mathsf{x}(g)$ in $\Om$. 
\end{prop}

 \begin{proof}
 As in the previous proof, we appeal to Theorem 10.1 in~\cite{gt}. Again, the key condition to be checked is (iii), meaning here that the function 
$$
b(x,u,v):=\frac{H'(u)H''(u)}{1+(H'(u))^2}|v|^2
$$
 is a non-increasing function of $u \in \R$ for any fixed $(x,v) \in \Om\times \Rn$. Upon computing when the derivative $\frac{d}{du} b(\cdot,u,\cdot)\leq 0$, we arrive at \eqref{cond-comp2}. 

Since $\mathsf{x}=H \circ \mathsf{y}$, the second assertion follows as well provided that $H$ is non-increasing.
\end{proof}

It turns out that the comparison principle holds for the $\mathsf{t}$ component function of a map as well. Indeed, by \eqref{harmonic-t} we define the following function $b:\Om\times \R \times \Rn \to \R$:
\begin{equation}
 b(x,u,v):=2\frac{\left(u+H(u)H'(u)\right) H''(u)} {1 + (H'(u))^2} |v|^2.  \label{t-coord-b}
\end{equation}
As in the last proof, we use Theorem 10.1 in~\cite{gt} and check its condition (iii) to verify that $b$ is a non-increasing function of $u \in \R$ for any fixed $(x,v) \in \Om\times \Rn$ provided that the following inequality holds:
\begin{equation}
  H'''(u+HH')(1+(H')^2)+(H'')^2(1-(H')^2)+H''(1+(H')^2-2uH') \leq 0. \label{cond-comp-pr-t}
\end{equation}

For the next result on the uniqueness of harmonic curves with the same boundary data we need to investigate if the conditions for the comparison principle to hold, are met for all component functions $\mathsf{x}, \mathsf{y}$ and 
$\mathsf{t}$ (the discussion on existence of harmonic curves is postponed till Proposition~\ref{prop-exist}). This amounts to checking that there are such $C^3$ functions $H$ so that conditions~\eqref{cond-comp2}, \eqref{cond-comp-pr-t} and $H'\leq 0$ hold together. The full discussion involves solving the system of three differential inequalities and leads to tedious and lengthy computations which do not fit into the scope of the manuscript. Instead we present the following example.

\begin{ex}
 Suppose that $H'''\equiv 0$ and so $H(u)=au^2+bu+c$ is a quadratic function and $a,b,c\in \R$. Then \eqref{cond-comp2} reads 
\[
(H'')^2(1-(H')^2)\leq 0\,\Leftrightarrow\, 1-(H')^2 \leq 0 \,\Leftrightarrow\, H'\leq -1\quad \hbox{(by the assumption that $H'\leq 0$)}.
\]
However, then \eqref{cond-comp-pr-t} takes the form ($H'''\equiv 0$)
\[
(H'')^2(1-(H')^2)+H''(1+(H')^2-2uH') \leq H''(1+(H')^2-2uH')
\]
and the latter expression is non-positive (under the condition $H'\leq -1$), for instance when $H''\leq 0$ and $1+(H')^2-2uH'\geq 0$ for all $u$. The latter two conditions hold for $a\leq 0$, provided that $b^2\leq 4a(a-1)$. Moreover, $H'\leq -1 \,\Leftrightarrow\, u\geq \frac{b+1}{-2a}$ and, therefore, it suffices to assume that 
 \begin{equation*}
   a\leq 0, \quad b^2\leq 4a(a-1), \quad b\leq -1-2a \min\{\inf_{\overline{\Om}}\mathsf{x}, \inf_{\overline{\Om}}\mathsf{y}, \inf_{\overline{\Om}}\mathsf{t}\},\quad c\in \R
 \end{equation*}
  for a quadratic function $H$ to imply the comparison principle for all component functions of a given $C^2$ harmonic curve.
 \end{ex}
 
\begin{cor}[Uniqueness of solutions]\label{cor-unique}
 Let $f$ and $g$ be a $C^2$-harmonic mappings defined on a domain $\Om\subset \Rn$, with the same boundary data in $C^0(\partial \Om)$. Moreover, let function $H$ be $C^3$, non-increasing and satisfy~\eqref{cond-comp2} and~\eqref{cond-comp-pr-t}. Then $f\equiv g$ in $\overline{\Om}$.
\end{cor}
The proof follows immediately from Proposition~\ref{comp-pr} applied to the first and second component functions of given harmonic mappings. This in turn together with the contact condition implies the uniqueness for the third component function. 

Similarly, we show the strong maximum principle.

\begin{prop}[strong maximum principle] 
 Let $f=(\mathsf{x}(f),\mathsf{y}(f),\mathsf{t}(f))$ be a $C^2$-harmonic mapping on a domain $\Om\subset \Rn$. If $\mathsf{y}(f)(x_0)=\sup_{\overline{\Om}} \mathsf{y}(f)$ at some $x_0 \in \Om$, then $\mathsf{y}(f)=const$ in $\Om$ provided that $H$ satisfies the following condition:
 \begin{equation}\label{cond-str-max}
 \sup_\Omega \left |\frac{(H''\circ \mathsf{y}) (H'\circ \mathsf{y})}{1 + (H' \circ \mathsf{y})^2 }\right | < C.
 \end{equation}
In such a case, also component functions $\mathsf{x}(f)$ and $\mathsf{t}(f)$ are constant in $\Om$. 
\end{prop}
\begin{proof}
 We apply the strong maximum principle in Theorem 5.3.1 in~\cite{puse-book} to the component function $\mathsf{y}(f)$ and, in notation of~\cite{puse-book}, with $B(x,z,\xi):=\frac{H'(z)\,H''(z)}{1+(H'(z))^2}|\xi|^2$.  Thus, we need to verify that \eqref{cond-str-max} implies assumptions (B1) and (F2) on pg. 107 in~\cite{puse-book}. Indeed:
\medskip
\\
\indent (B1) reads $B(x,z,\xi)\geq -\kappa \Phi(|\xi|)-f(z)$ which in our case holds with $\kappa:=C$ for $C$ as in~\eqref{cond-str-max}, $\Phi(|\xi|):=|\xi|^2$ and $f\equiv 0$,
\smallskip
\\
\indent (F2) requires that $f(0)=0$ and $f$ is non-decreasing on some interval and for us holds trivially, as $f\equiv 0$.
\medskip
\\
 Thus, we obtain that $\mathsf{y}(f)=const$ in $\Om$.  Since $\mathsf{x}(f)=H\circ \mathsf{y}$ and, moreover, $\mathsf{t}(f)$ satisfies the contact condition, we infer that also $\mathsf{x}(f)$ and $\mathsf{t}(f)$ must be constant under assumptions of the proposition.
\end{proof}

%
%
%

We close this section with the proof of the Harnack inequality for component functions of a harmonic curve.
\begin{prop} 
 Let $f=(\mathsf{x}(f),\mathsf{y}(f),\mathsf{t}(f))$ be a $C^2$-harmonic mapping defined on a domain $\Om\subset \Rn$ such that~\eqref{cond-str-max} holds for some $C>0$ and $\mathsf{y}(f) \geq 0$ in $\Om$. Then $\mathsf{y}(f)$ satisfies the Harnack inequality in any ball $B_r\subset \Om$:
\[
 \sup_{B_r} \mathsf{y}(f)\leq c \inf_{B_r} \mathsf{y}(f),
\] 
where the positive constant $c=c(n, C, \|\mathsf{y}(f)\|_{L^{\infty}(\Om)})$. Moreover, the analogous Harnack inequality holds for $\mathsf{x}(f)$ and $\mathsf{t}(f)$.
 \end{prop} 
\begin{proof}
 We employ results presented in~\cite{mz}. Note that \eqref{criticaly}, under the assumption~\eqref{cond-str-max}, 
 written in the divergence form satisfies the (standard) growth assumptions for $(A,B)$-harmonic equations as in (3.5), see pg. 162 in~\cite{mz}, i.e. for $A(x, z, \xi)=\xi$ and $B(x, z, \xi)=\frac{H'(z)\,H''(z)}{1+(H'(z))^2}|\xi|^2$ it holds that
 \[
  |A(x, z, \xi)|\leq |\xi|,\quad |B(x, z, \xi)|\leq c |\xi|^2.
 \]
 Therefore, Theorem 3.14 in~\cite{mz} applies with $k(r)\equiv 0$ and a positive constant $C$ as in Theorem 3.13 in \cite{mz} such that $c=c(n, C, \|\mathsf{y}(f)\|_{L^{\infty}(B_r)})$, giving the assertion, as the $C^2$-function $\mathsf{y}(f)$ is locally bounded on any domain in $\Rn$.
 
 In order to show the Harnack estimate for $\mathsf{x}(f)$ we note that by~\eqref{eq-f1} in any compact subset of $\Om$  it holds that $\Delta \mathsf{x}(f)=\tilde{B}(x,\mathsf{y}(f), \nabla \mathsf{y}(f))$ for a locally bounded function 
 $$
 \tilde{B}(x)=\frac{H''\big(\mathsf{y}(f)(x)\big)}{1+\big(H'(\mathsf{y}(f)(x))\big)^2}|\nabla \mathsf{y}(f)(x)|^2.
 $$
The function $\tilde{B}$ depends only on a point $x\in \Om$, as the function $\mathsf{y}(f)$ is a given solution to~\eqref{criticaly}. The local boundedness of $\tilde{B}$ is trivial consequence of $C^2$-regularity of functions $\mathsf{y}(f)$ and $H$. Thus, we may apply e.g. Theorem 7.2.1 in~\cite{puse-book} to obtain the Harnack inequality for $\mathsf{x}(f)$.

Finally, the Harnack estimate for the component function $\mathsf{t}(f)$ follows by the reasoning analogous to the above one for $\mathsf{y}(f)$. Namely, as $\mathsf{t}(f)$ satisfies~\eqref{harmonic-t}, we may apply again Theorem 3.14 in~\cite{mz} with operator $b$ as defined in~ \eqref{t-coord-b}. 
\end{proof}

\subsection{The existence of solutions}

Recall that Theorem 3.8 in~\cite{cl} gives us the existence of the Sobolev minimizers to the Dirichlet problem subject the Sobolev boundary data in the sense of traces, cf.~\cite[Definition 3.1 ]{cl}. Moreover, Theorems 3.2 and 3.3 in~\cite{cl} provide equivalent Dirichlet problems in terms of the horizontal energy $\int |\nabla \mathsf{z}|^2$, i.e. with respect to the horizontal part of the Jacoby matrix of a harmonic map $f=(\mathsf{z},\mathsf{t})$. However, in this work we study a slightly different class of mappings, namely those which are the critical points of the horizontal energy and are $C^2$ regular.

\begin{prop}\label{prop-exist}
 Let $\Om\subset \Rn$ be a domain such that $\partial \Om$ satisfies the exterior sphere condition everywhere. Then for any function $g \in C^{0}(\partial \Om, \R)$ and a $C^3$ function $H$ satisfying~\eqref{cond-comp2} there exists a unique solution $\mathsf{y}\in C^0(\overline{\Om})\cap C^2(\Om)$ to the following Dirichlet problem (cf. ~\eqref{criticaly})
\begin{equation*}
 \begin{cases}
\Delta \mathsf{y}  = - \frac{(H''\circ \mathsf{y}) (H'\circ \mathsf{y})}{1 + (H' \circ \mathsf{y})^2 } |\nabla \mathsf{y}|^2 \quad\hbox{ in }\Om\\
\mathsf{y}=g \quad \hbox{ on }\partial \Om.
\end{cases}
\end{equation*} 
\end{prop}

\begin{proof}
The existence of solution follows from Theorem 15.18 in~\cite{gt} applied to the following elliptic operator $Q$ (as in~\eqref{criticaly}) upon checking that hypotheses of Theorems 15.5, 14.1 and condition (10.36) in \cite{gt} hold:
\[
Q(u):= \Delta u + \frac{(H''(u) (H'(u)}{1 + (H'(u))^2} |\nabla u|^2. 
\]
  Indeed, checking hypotheses of Theorem 15.5 amounts to tedious verification of assumptions (15.53): in the notation of ~\cite{gt} we have that $\Lambda=\lambda=1$ the ellipticity constants for the Laplacian, $r=-1$, $s=0$ and so, since the coefficients of the principal part of $Q$ are simply $a_{ij}=\delta_{ij}$, the first line of assumptions in (15.53) holds trivially. For the second line in (15.53) we first observe that for us the lower order terms part of the operator $Q$ satisfies the estimate: $b(x, u,p)\leq C |p|^2$. Moreover, recall the differential operators used in \cite{gt} to verify (15.53): 
 \[
 \delta:=D_z+\sum_{i}\frac{p_i}{|p|} D_{x_i}\,\, (\hbox{see (15.8)}),\quad \overline{\delta}:=\sum_{i}p_i D_p\,\, (\hbox{see (15.19)}),\quad \partial_i:=D_{p_i}.
 \]
  Then, the direct differentiation allows us to check the remaining part of assumptions in (15.53) (with $\theta=2$).
   Similarly, checking hypotheses of Theorem 14.1 reduces to checking condition (14.19) on pg. 337 in \cite{gt} (see the  discussion following the statement of Theorem 14.1).
  Finally (10.36) reads: $zb(x,z,0)\leq 0$ and is trivially satisfied, as for us $b(x,z,0)\equiv 0$. 
 
The uniqueness of solution is an immediate consequence of Corollary~\ref{cor-unique}.
\end{proof}

Similarly we get the existence of solutions to the Dirichlet problem with continuous boundary data for $\mathsf{x}$ and $\mathsf{t}$, by \eqref{eq-f1} and~\eqref{harmonic-t}, respectively.

\subsection{The Phragm\`en--Lindel\"of theorem}

Let us notice that if $\mathsf{y}$ solves equation~\eqref{criticaly}, then trivially it is also a harmonic subsolution, meaning that $\Delta \mathsf{y} \geq 0$, provided that $H'(s)\,H''(s)\leq 0$ for all $s\in \R$. Furthermore, recall that by direct computations a Euclidean norm of a point $x\in \Rn$ is subharmonic:
$$
\Delta |x|^\alpha=\alpha(\alpha-2)|x|^{\alpha-2}\leq 0,\quad \hbox{for }\quad 0\leq \alpha\leq 2.
$$
Hence, we may directly apply Theorem 19 in~\cite[Section 9]{pw} to obtain the following variant of the Phragm\`en--Lindel\"of theorem.

Let $\Om\subset \Rn$ be an open connected unbounded set and $\Gamma\subset \partial \Om$ be a portion of the boundary of $\Om$. Suppose that there exists an increasing sequence of bounded domains $\Om_1\subset \Om_2\subset\cdots\subset \Om_l\cdots$ such that $\Om_l\subset \Om$ for any $l=1,2,\ldots$ with two properties:
\begin{itemize}
\item[($\Om1$)] For any $x\in \Om$ there exists $l_0$ such that $x\in \Om_{l_0}$ (and so also for all $l\geq l_0$);
\item[($\Om2$)] For each $l=1,2,\ldots$ it holds that $\partial \Om_l=\Gamma_l\cup \Gamma_l'$, where $\Gamma_l\subset \Gamma$ and $\Gamma_l'\subset \Om$.
\end{itemize}
A classical example of domains satisfying the above conditions is provided by a half-space $\Om=\R^m_{+}$ and half-balls of radius $l$ centered at the origin, i.e.  $\Om_l=B_l\cap \Om$ for $l=1,2,\ldots$. Here $\Gamma=\partial \Om=\R^{m-1}$.

\begin{prop}\label{prop-PL}
 Let $f=(\mathsf{x}(f),\mathsf{y}(f),\mathsf{t}(f))$ be a $C^2$-harmonic mapping defined on an unbounded domain $\Om\subset \Rn\setminus\{0\}$ such that
$$
\mathsf{y}(f)\leq 0 \quad\hbox{on a subset of a boundary }\Gamma\subset \partial \Om \quad \hbox{ and that }\quad H'(s)\,H''(s)\leq 0 \quad\hbox{for all }s\in \R.
$$
Furthermore, let $(\Om_k)$ be a sequence of bounded domains in $\Om$ satisfying conditions ($\Om1$) and ($\Om2$) above.

Then either $\mathsf{y}(f)\leq 0$ in $\Om$ or $\mathsf{y}(f)$ satisfies the following growth condition:
 \[
  \liminf_{l\to \infty} \Big(\sup_{\Gamma_l'} \frac{\mathsf{y}(f)(x)}{|x|^\alpha} \Big) \leq 0,\quad 0\leq \alpha \leq 2.
 \]
\end{prop}

\begin{proof}
 The proof follows by reduction to the harmonic Phragm\`en--Lindel\"of theorem. Apply Theorem 19 in~\cite[Chapter 2, Section 9]{pw} take $L=\Delta$, the Laplacian, function $h\equiv 0$ and functions $w=w_k:=|x|^\alpha$ for all $k=1,2,\ldots$.
\end{proof}

Suppose that a $C^2$ function $\mathsf{y}(f)$ satisfying~\eqref{criticaly} is given on an unbounded open set $\Om\subset \Rn\setminus\{0\}$ and there exist a sequence of bounded domains $(\Om_k)$ in $\Om$ satisfying conditions ($\Om1$) and ($\Om2$) above. Then, similar assertions as in Proposition~\ref{prop-PL} can be proven for component functions $\mathsf{x}(f)$ and $\mathsf{t}(f)$ by employing again Theorem 19 in~\cite[Chapter 2, Section 9]{pw}. Indeed, by~\eqref{eq-f1} in the notation of~\cite{pw} this equation can be reformulated as follows
\[
L+h(x):=\Delta \mathsf{x}-\frac{H''\circ \mathsf{y}}{1+(H'\circ \mathsf{y})^2}|\nabla \mathsf{y}|^2.
\]
By assuming that $H''\leq 0$, we have that $L+h\geq 0$ and~\cite{pw} applies. Analogous growth condition on $H, H'$ and  $H''$, implied by~\eqref{harmonic-t}, allow us to infer the Phragm\`en--Lindel\"of theorem for $\mathsf{t}(f)$ as well.

\subsection{The three spheres theorem}

Following the same approach as in the case of the Phragm\`en--Lindel\"of theorem based on reduction of the discussion to the case of harmonic subsolutions, we prove the following variant of the three-spheres theorem.

\begin{prop}
 Let $f=(\mathsf{x}(f),\mathsf{y}(f),\mathsf{t}(f))$ be a $C^2$-harmonic mapping defined on a domain $\Om\subset \Rn$ such that $ H'(s)\,H''(s)\leq 0$ for all $s\in \R$. Then, it holds for two concentric balls $B(0, r_1)\Subset B(0, r_2)\subset \Rn$ and the annular region between them, that for all $r_1<r<r_2$ the following inequality holds for component function $\mathsf{y}(f)$:
 \[
 M(r) \leq  M(r_1)\,\frac{r^{2-n}-r_2^{2-n}}{r_1^{2-n}-r_2^{2-n}} + M(r_2)\,\frac{r_1^{2-n}-r^{2-n}}{r_1^{2-n}-r_2^{2-n}},
 \]
 where $M(r):= \sup_{|x|=r} \mathsf{y}(f)(x)$. The similar inequality holds for $\mathsf{x}(f)$ provided that $H''\geq 0$ in $\R$. Moreover, if $H$ satisfies the differential inequality
 \begin{equation}\label{t-cond-3sph}
  (s +H(s) H'(s))H''(s) \geq 0,
 \end{equation}
 then the similar assertion holds for  $\mathsf{t}(f)$.
\end{prop}

\begin{proof}
 We reduce the discussion to the harmonic three-spheres theorem by applying Theorem 30 in \cite[Section 12]{pw} and by noticing that under the assumptions on $H$, it holds that $\Delta \mathsf{y}(f) \geq 0$ as $\mathsf{y}(f)$ satisfies~\eqref{criticaly}. Moreover, if $H''\geq 0$ then by \eqref{eq-f1} we have that $\Delta \mathsf{x}(f)\geq 0$ and we apply again Theorem 30 in~\cite[Section 12]{pw} obtaining the assertion for $\mathsf{x}(f)$.
Finally, by \eqref{harmonic-t}, if condition~\eqref{t-cond-3sph} holds, then we have $\Delta \mathsf{t}(f)\geq 0$ and again obtain the assertion for $\mathsf{t}(f)$.
\end{proof}

%
%

\end{document}